\documentclass{amsart}

\usepackage{paralist}
\usepackage{color}
\usepackage{mathrsfs}
\usepackage{amssymb}
\usepackage{amsmath}
\usepackage{amsfonts}
\usepackage{stmaryrd}
\usepackage[backend=biber]{biblatex}
\bibliography{references.bib}

\newtheorem*{thma}{Theorem~A}
\newtheorem*{thmb}{Theorem~B}

\newtheorem{theorem}{Theorem}[section]
\newtheorem{lemma}[theorem]{Lemma}
\newtheorem{proposition}[theorem]{Proposition}
\newtheorem{corollary}[theorem]{Corollary}
\newtheorem{claim}[theorem]{Claim}

\newtheorem{question}[theorem]{Question}
\newtheorem{example}[theorem]{Example}

\theoremstyle{definition}
\newtheorem{definition}[theorem]{Definition}
\newtheorem{remark}[theorem]{Remark}

\usepackage[pdftex]{hyperref}
\hypersetup{
    colorlinks=true, 
    linktoc=all,     
    linkcolor=blue,  
    allcolors=blue,
}

\usepackage[framemethod=tikz]{mdframed}

\newcommand{\dom}{\mathrm{dom}}
\newcommand{\bb}{\mathbb}

\newcommand{\power}{\ensuremath{\mathscr{P}}}

\newcommand{\Add}{\mathrm{Add}}

\newcommand{\R}{\bb{R}}
\newcommand{\Q}{\bb{Q}}
\renewcommand{\P}{\bb{P}}

\newcommand{\U}{\mathcal{U}}

\newcommand{\ZF}{\sf ZF}

\newcommand{\PFA}{\sf PFA}

\newcommand{\cl}{\mathrm{cl}}
\newcommand{\F}{\mathcal{F}}

\newcommand{\w}{\omega}
\newcommand{\seq}[1]{{\langle{#1}\rangle}}

\renewcommand{\subset}{\subseteq}
\renewcommand{\epsilon}{\varepsilon}

\renewcommand{\phi}{\varphi}

\title{Preservation of some topological properties under forcing}

\begin{document}

\author{Chris Lambie-Hanson}
\address{Institute of Mathematics, 
Czech Academy of Sciences, 
{\v Z}itn{\'a} 25, Prague 1, 
115 67, Czech Republic
}
\email{lambiehanson@math.cas.cz}

\author{Pedro Marun}
\address{Institute of Mathematics, 
Czech Academy of Sciences, 
{\v Z}itn{\'a} 25, Prague 1, 
115 67, Czech Republic}
\email{marun@math.cas.cz}

\subjclass[2020]{03E40, 54D20, 54D99}
%

\keywords{Topology, forcing, strong properness, countable tightness, games}
\thanks {Both authors were supported by the Czech Academy of Sciences (RVO 67985840) and the GA\v{C}R project 23-04683S}

\begin{abstract}
We add to the theory of preservation of topological properties under forcing. In particular, we answer a question of Gilton and Holshouser in a strong sense, showing that if player II has a winning strategy in the strong countable fan tightness game of a space at a point, then this continues to hold in every set forcing extension of the universe. The same is also true for the Rothberger game, but not for the countable fan tightness or Menger games.
\end{abstract}

\maketitle

\section{Introduction}

Suppose that $V\subset W$ are transitive models of set theory and that $(X,\tau)\in V$ is such that $V\models $``$(X,\tau)$ is a topological space''. From the point of view of $W$, $(X,\tau)$ need not be a topological space, because $W$ may contain new subsets of $\tau$ whose union is not in $\tau$. On the other hand, from the point of view of $W$, $\tau$ is a basis for some topology on the set $X$, and one can ask how this topology compares with $\tau$. In practice, one often starts with a topological space $(X,\tau)$ satisfying a property $\Phi$ in some ground model $V$, passes to a generic extension by some forcing $\P$, and asks whether the topology generated by $\tau$ in $V^\P$ satisfies property $\Phi$. In general, there is no straightforward answer to this question; the situation depends on the space $(X,\tau)$, the property $\Phi$ and the forcing poset $\Phi$. The following examples illustrate this:

\begin{example}[{\cite[Example 4.1]{giltonPreservationTopologicalProperties2025}}]
Let $X$ be the closed unit interval $[0,1]$ with the Euclidean topology. If $\P$ is any forcing that adds a new real, then $X$ will not be compact in $V^\P$, for it contains all the rationals between $0$ and $1$ and is therefore not closed as a subset of the closed unit interval of the generic extension, which is a compact Hausdorff space in $V^\P$.
\end{example}

A less trivial example:

\begin{example}[Burke, {\cite[Theorem 17]{tallCardinalityLindelofSpaces1995}}]\label{X}
The space $X={}^{\omega_1} 2$ with the product topology\footnote{where each copy of $2$ is given the discrete topology} is compact by Tychonoff's theorem, but is not even Lindel\"of in the generic extension by $\Add(\omega_1,1)$. Indeed, let $\dot g$ be a name for the generic function $\omega_1\to 2$. For each $\alpha<\omega_1$, let $\dot U_\alpha$ be a $\P$-name for the set
\[
\{x\in X:x(\alpha) = \dot g(\alpha)\}.
\]
Observe that the interpretation of $\dot U_\alpha$ by any generic filter is a ground-model open set, hence it is forced that $\dot U_\alpha$ is open in $X$ (with respect to the new topology). An easy genericity argument shows that $\Vdash X=\bigcup_{\alpha<\omega_1}\dot U_\alpha$, but no countable sub-collection of $\{\dot U_\alpha:\alpha<\omega_1\}$ can cover $X$.
\end{example}

The preservation of the Lindel\"of property under countably closed forcings has received the most attention in the published literature, so much so that ``indestructible Lindel\"of'' usually means a Lindel\"of space which remains so after any countably closed forcing. For a wealth of information on this matter, see \cite{tallCardinalityLindelofSpaces1995} and \cite{scheepersLindelofIndestructibilityTopological2010}.

In this paper, we consider the preservation of game-theoretic versions of classical topological 
properties, namely the Menger and Rothberger properties, which strengthen the Lindel\"{o}f 
property, and (strong) countable fan tightness, a strengthening of countable tightness.
In \cite[Question 8.1]{giltonPreservationTopologicalProperties2025}, Gilton and Holshouser ask whether player II having a winning strategy in the strong countable fan tightness game is preserved by strongly proper forcings (for undefined notions, see section 2). Our main theorem shows that the answer is ``Yes'', in a very strong sense:

\begin{thma}
Suppose that $(X, \tau)$ is a topological space and player II has a winning strategy in the strong countable fan tightness game for $X$. Then this is preserved in any forcing extension of $V$.
\end{thma}

The same proof will give the following theorem, which strengthens \cite[Theorem 7.2]{giltonPreservationTopologicalProperties2025}:

\begin{thmb}
Suppose that $(X, \tau)$ is a topological space and player II has a winning strategy in the Rothberger game for $X$. Then this is preserved in any forcing extension of $V$.
\end{thmb}

The crucial difference between the games mentioned in Theorems A and B and the closely related 
countable fan tightness and Menger games lies in the fact that, in the former games, player II 
plays a single point or open set, respectively, on each move, whereas in the latter, they play 
a collection of points or open sets of arbitrary finite size on each move. As we will see, due 
to this difference, the analogues of Theorems A and B for the countable fan tightness and Menger 
games provably fail.

\section{Preliminaries}

If $(X,\tau)$ is a topological space and $\P$ is a forcing notion, then we let $(X,\tau^\P)$ denote (a name for) the space with underlying set $X$ and topology generated by $\tau$ in the generic extension $V^\P$. We will abuse notation and speak of ``the space $X$'', with the understanding that the model in which the topological properties are evaluated can be deduced from context. For example, when we write ``$\Vdash_\P X$ has property $\Phi$'', we really mean $X$ as computed in the $\P$-generic extension, i.e. that it is forced that $(X,\tau^\P)$ has property $\Phi$. One way around these technicalities, suggested by Dow in \cite{dowTwoApplicationsReflection1988}, is to redefine the notion of a topological space to mean a pair $(X,\mathcal{B})$, where $\mathcal{B}$ is a basis for a (unique) topology on $X$. Being a basis is absolute for transitive models of $\ZF$, which renders mute the notational considerations above, but we will follow the more common, if slightly ambiguous, terminology and definitions.

If $\kappa$ and $\lambda$ are cardinals, with $\kappa$ infinite, then $\Add(\kappa,\lambda)$ is the poset for adding $\lambda$-many Cohen subsets of $\kappa$. More precisely, $\Add(\kappa,\lambda)$ consists of the set of partial functions $\lambda \times\kappa \rightharpoonup 2$ of size less than $\kappa$. We order $\Add(\kappa,\lambda)$ by (reverse) inclusion.

Let $\P$ be a forcing poset. We will say that a regular cardinal $\theta$ is \emph{sufficiently large} (for $\P$) iff $\power(\P)\in H_\theta$. If other objects, like a topological space $(X,\tau)$, are at play, the phrase ``sufficiently large'' will also imply that these objects belong to $H_\theta$. 

The following definition goes back to Mitchell in \cite{mitchellAddingClosedUnbounded2005}:

\begin{definition}\label{stronglyProper}
Let $\P$ be a forcing poset and $\theta$ a sufficiently large regular cardinal. Given a well-ordering $\triangleleft$ of $H_\theta$, a model $M\prec (H_\theta,\in,\triangleleft)$ and a condition $q\in \P$, we say that $q$ is \emph{strongly $(M,\P)$-generic} (also called a \emph{strong master condition for $M$}) iff for every set $D\subset \P\cap M$ which is dense in the poset $\P \cap M$, $q\Vdash_\P \check D \cap \dot G\neq\emptyset$, where $\dot G$ is the canonical $\P$-name for the generic filter. When the model $M$ and/or the poset $\P$ are clear from context, we will drop reference to them.

Given $M\prec H_\theta$, we say that $\P$ is \emph{$M$-strongly proper} iff for all $p\in \P\cap M$ there exists $q\in \P$ so that $q\le p$ and $q$ is strongly $(M,\P)$-generic.

Given a non-empty collection $\mathcal{S}$ of elementary submodels of $H_\theta$ with $\P\in\bigcap\mathcal{S}$, we say that $\P$ is \emph{$\mathcal{S}$-strongly proper} iff $\P$ is $M$-strongly proper for all $M\in\mathcal{S}$.
\end{definition}

For more on strong properness, in particular its comparison with Shelah's notion of properness, see \cite[Section 3]{neemanForcingSequencesModels2014} and \cite[Section 3]{giltonPreservationTopologicalProperties2025}.

Recall that a topological space $(X,\tau)$ is a \emph{Lindel\"of space} iff every open cover has a countable subcover. We will be concerned with the following two strengthenings of the Lindel\"of property:

\begin{itemize}
\item A topological space $(X,\tau)$ has the \emph{Rothberger property} (or simply ``is Rothberger'') iff for every sequence $\seq{\U_n:n<\omega}$ of open covers of $X$, there exists a sequence $\seq{U_n:n<\omega}$ with $U_n\in\U_n$ for every $n\in\omega$ so that $\{U_n:n<\omega\}$ covers $X$.
\item A topological space $(X,\tau)$ has the \emph{Menger property} (or simply ``is Menger'') iff for every sequence $\seq{\U_n:n<\omega}$ of open covers of $X$, there exists a sequence $\seq{\mathcal{F}_n:n<\omega}$ with $\mathcal{F}_n\in [\U_n]^{<\omega}$ for every $n\in\omega$ so that $\{\bigcup\F_n:n<\omega\}$ covers $X$.\footnote{for a set $A$, $[A]^{<\omega}$ denotes the family of finite subsets of $A$.}
\end{itemize}

Clearly, every compact space is Menger, but being Rothberger is a stronger property:

\begin{example}
Let $X$ be the closed unit interval with the Euclidean topology. Choose a sequence $\seq{r_n:n<\omega}$ of positive numbers with $\sum_n r_n<1$. For each $n\in\omega$, let $\U_n$ be the collection of open subsets of $X$ which have Lebesgue measure at most $r_n$. Clearly, $\U_n$ covers $X$. If $U_n\in\U_n$ for every $n$, then $\bigcup_n U_n$ has measure at most $\sum_n r_n<1$, hence does not cover $X$.
Thus, $X$ is compact but not Rothberger.
\end{example}

\begin{definition}\label{fan}
Let $(X,\tau)$ be a topological space and $x\in X$. We say that
\begin{itemize}
\item $(X,\tau)$ has \emph{strong countable fan tightness at $x$} iff for every sequence $\seq{A_n:n<\omega}$ of subsets of $X$ with $x\in\cl(A_n)$ for every $n\in\omega$, there exists a sequence $\seq{x_n:n<\omega}$ with $x_n\in A_n$ for every $n<\omega$ and $x\in \cl(\{x_n:n<\omega\})$.
\item $(X,\tau)$ has \emph{countable fan tightness at $x$} iff for every sequence $\seq{A_n:n<\omega}$ of subsets of $X$ with $x\in\cl(A_n)$ for every $n\in\omega$, there exists a sequence $\seq{F_n:n<\omega}$ with $F_n\in [A_n]^{<\w}$ for every $n<\omega$ and $x\in \cl(\bigcup_n F_n)$.
\end{itemize}
\end{definition}

We will now introduce a series of game analogues of the aforementioned topological properties. For a survey on the topic of topological games, see \cite{telgarskyTopologicalGames50th1987}.

If $(X,\tau)$ is a topological space and $x\in X$, then the \emph{strong countable fan tightness game for $X$ at the point $x$} is a two-player game of length $\omega$ that proceeds as follows. In each round, player I first plays a set $A_n \subseteq X$ such that $x \in \mathrm{cl}(A_n)$; then player II plays a point $x_n \in A_n$. Player II wins iff $x \in \mathrm{cl}(\{x_n :n <\omega\})$; otherwise, player I wins.

If $(X,\tau)$ is a topological space and $x\in X$, then the \emph{countable fan tightness game for $X$ at the point $x$} is a two-player game of length $\omega$ that proceeds as follows. In each round, player I first plays a set $A_n \subseteq X$ such that $x \in \mathrm{cl}(A_n)$; then player II plays a finite set $F_n\subseteq A_n$. Player II wins iff $x \in \mathrm{cl}(\bigcup_{n<\omega} F_n)$; otherwise, player I wins.

If $(X,\tau)$ is a topological space, then the \emph{Rothberger game for $X$} is a two-player game of length $\omega$ that proceeds as follows. In each round, player I first plays an open cover $\mathcal{U}_n$ of $X$; then player II plays an open set $U_n\in \mathcal{U}_n$. Player II wins iff $\{U_n:n<\omega\}$ covers $X$; otherwise player I wins.

If $(X,\tau)$ is a topological space, then the \emph{Menger game for $X$} is a two-player game of length $\omega$ that proceeds as follows. In each round, player I first plays an open cover $\mathcal{U}_n$ of $X$; then player II plays a finite subset $\mathcal{F}_n \subset \mathcal{U}_n$. Player II wins iff $\{\bigcup \mathcal{F}_n:n<\omega\}$ covers $X$; otherwise player I wins.

We have the following well known fact.

\begin{lemma}\label{game}
Let $(X,\tau)$ be a topological space. Suppose that player II has a winning strategy in the Rothberger (respectively Menger, strong countable fan tightness, countable fan tightness) game. Then $(X,\tau)$ is Rothberger (respectively is Menger, has strong countable fan tightness, has countable fan tightness).
\end{lemma}

\begin{proof}
We prove the lemma for the Menger property, the other cases are analogous. Let $\sigma$ be a winning strategy for player II in the Menger game. Given $\seq{\U_n:n<\omega}$ a sequence of open covers of $X$, consider the run of the game where I plays $\U_n$ on round $n$. Let $\F_n$ be the play by II at round $n$ according to $\sigma$. Then $\bigcup_n\F_n$ covers $X$, because $\sigma$ is winning, and so $X$ is Menger.
\end{proof}

For any undefined set-theoretic notions, see \cite{jechSetTheory2003} or \cite{kunenSetTheory2011}. For any undefined topological notions, see \cite{engelkingGeneralTopology1989}.

\section{Strong Properness}

The following proposition was proven by the second author in his PhD thesis (see \cite[Theorem 3.32]{marunSpacesTreesLabelled2024}), and independently and almost at the same time by Gilton and Holshouser (see \cite[Theorem 4.6]{giltonPreservationTopologicalProperties2025}).

\begin{proposition}\label{strong}
Let $\P$ be a forcing notion, $(X,\tau)$ a topological space, and $\theta$ a sufficiently large 
regular cardinal. Let $\mathcal{S}$ be a stationary subset of $[H_\theta]^\omega$ and suppose that $\P$ is $\mathcal{S}$-proper. If $(X,\tau)$ is Lindel\"of, then $\Vdash_\P X$ is Lindel\"of.
\end{proposition}

The following lemma appeared, in greater generality, in \cite[Lemma 1.12]{giltonSIDECONDITIONSITERATION2016}. The proof presented there has a very small gap, so we will use this opportunity to reprove the case that is of interest to us:

\begin{lemma}[Gilton-Neeman]\label{GN}
Let $\kappa\ge\omega_1$ and set $\P:=\Add(\omega,\kappa)$. Let $\dot\Q$ be a $\P$-name for a $\sigma$-closed forcing. Let $\theta$ be a sufficiently large (for $ \P\ast\dot\Q$) regular cardinal and $M\prec H_\theta$ be countable. Then the two-step iteration $\P\ast\dot\Q$ is $M$-strongly proper.
\end{lemma}

\begin{proof}
Fix a surjection $\phi:\w \twoheadrightarrow M$ and a condition $(p_0,\dot{q}_0)\in M\cap (\P\ast\dot\Q)$. Choose an ordinal $\alpha\in \omega_1\setminus M$ and let $\dot f$ be a $\P$-name for the $\alpha$th real added by $\P$, i.e. $\Vdash_\P \dot f = \bigcup \{p(\alpha,\cdot): p\in \dot G\}$, where $\dot G$ is a name for the $\P$-generic filter.

Recursively define $\seq{\dot q_i:i<\omega}$ as follows (with $\dot q_0$ already defined): given $\dot q_i$, by the maximum principle we can find a name $\dot q_{i+1}$ for a condition in $\dot\Q$ so that
\[
\Vdash_\P \left[\left(\phi(\dot f (i) )\le \dot q_i \to \dot q_{i+1}= \phi(\dot f(i))\right )\wedge \left(\phi(\dot f (i) )\not\le \dot q_i \to \dot q_{i+1}= \dot q_i\right )\right].
\]
Note that, for every $i<\omega$, we have that $\Vdash \dot q_i \in M[\dot G]$, but the names $\dot q_i$ themselves need not belong to $M$. However, we have a `local' version of this:
\begin{claim}
For every $n\in\omega$ and every $p\in\P$ with $\dom(p(\alpha,\cdot)) \supseteq n$, there exists a $\P$-name $\dot q\in M$ for a condition in $\dot \Q$ such that $p\Vdash \dot q=\dot q_n$.
\end{claim}

\begin{proof}
The proof is by induction on $n$, with the case $n=0$ being trivial. Suppose the lemma holds for $ n$ and let $p\in\P$ be such that $n+1\subset \dom(p(\alpha,\cdot))$. By the induction hypothesis, we can find $\dot q_n'\in M$, a $\P$-name for a condition in $\dot\Q$, so that $p\Vdash {\dot{q}}_n'=\dot{q}_n$. Since $n\in \dom(p(\alpha,\cdot))$, letting $m:=p(\alpha,n)$ we have that $p\Vdash \dot f(n)=m$. Note that
\[
\Vdash_\P \exists q\in\dot \Q\left[\left(\phi(m)\le \dot q_n'\to \dot q = \phi(m)\right) \wedge \left(\phi(m)\not\le \dot q_n' \to \dot q = \dot q_n'\right)\right].
\]
Since all the parameters belong to $M$, by elementarity we can find $\dot q'_{n+1}\in M$ so that $\Vdash_\P \dot q'_{n+1}\in\dot\Q$ and
\[
\Vdash_\P \left[\left(\phi(m)\le \dot q'_n\to \dot q'_{n+1} = \phi(m)\right) \wedge \left(\phi(m)\not\le \dot q'_n \to \dot q'_{n+1} = \dot q_n'\right)\right].
\]
Since $p\Vdash \dot q'_{n}=\dot q_n$, it follows that $p\Vdash \dot q'_{n+1}=\dot q_{n+1}$.
\end{proof}
The sequence of names $\seq{\dot q_i:i<\omega}$ is forced to be decreasing, so by countable closure we can find a $\P$-name for a condition $\dot q_\omega$ forced to be a lower bound of $\seq{\dot q_i:i<\w}$. We will argue that $(p_0,\dot q_\w)$ is strongly generic for $M$. To see this, fix $D$ a dense subset of $(\P\ast\dot\Q)\cap M$, we will argue that every extension of $(p_0,\dot q_\w)$ can be further extended to force $D \cap \dot K\neq\emptyset$, where $\dot K$ is a name for the $\P\ast\dot\Q$-generic filter. Fix therefore $(s,\dot r)\le (p_0,\dot q_\w)$ and set $s_0:=s\cap M$, so that $s_0\in M$ by elementarity. Applying the claim to $s$ and $n:=\min\{k\in\omega:(\alpha,k)\not\in\dom(s)\}$, we can find $\dot q'_{n}\in M$ a $\P$-name for an element of $\dot\Q$ which satisfies $s\Vdash \dot q'_n=\dot q_n$. In particular, $(s_0,\dot q'_n)$ is a condition in $(\P\ast\dot\Q)\cap M$, so, by the density of $D$, we can find $(t,\dot u)\in D$ with $(t,\dot u)\le (s_0,\dot q'_{n})$. Since $\phi$ is a surjection, we can fix $m\in\w$ with $\phi(m)=\dot u$. 

Set $s':=s\cup t \cup \{(\alpha,n,m)\}$. Since $t\le s_0=s\cap M$ and $(\alpha,n)\not\in\dom(s)$, $s'$ is a condition extending $s$. Therefore, $(s',\dot r)$ is a condition in $\P\ast\dot\Q$. We claim that it, moreover, extends $(t,\dot u)$. Obviously, $s'\le t$, so we only need to argue that $s'\Vdash \dot r\le \dot u$. Note that
\begin{itemize}
\item $s'\Vdash \dot f(n)=m$;
\item $t\Vdash \dot u \le \dot q_n'$;
\item $s\Vdash \dot q_n'=\dot q_n$;
\end{itemize} 
so, by the choice of $m$,
\[
s'\Vdash \phi(f(n))=\phi(m)=\dot u \le \dot q_n.
\]
Therefore, $s'\Vdash \dot q_{n+1}=\dot u$, and hence
\[
s'\Vdash \dot r \le \dot q_\omega\le \dot q_{n+1} = \dot u.
\]
We have shown that $(s',\dot r)\le (t,\dot u)$. Since $(t,\dot u)\in D$, in particular $(s',\dot r)\Vdash D\cap \dot K\neq\emptyset$, as desired.
\end{proof}

Combining Proposition \ref{strong} and Lemma \ref{GN}, we immediately obtain the following result, first proved by Dow in \cite[Lemma 3.3]{dowTwoApplicationsReflection1988}, and later strengthened (by a different method) by Scheepers and Tall in \cite[Corollary 12]{scheepersLindelofIndestructibilityTopological2010}

\begin{corollary}[Dow]\label{dow}
Let $(X,\tau)$ be a Lindel\"of space. Then, in $V^{\Add(\omega,\omega_1)}$, $X$ is a Lindel\"of space which is moreover indestructible under countably closed forcing.
\end{corollary}

\begin{proof}
The fact that $X$ is Lindel\"of in $V^{\Add(\omega,\omega_1)}$ follows by Proposition \ref{strong}, because $\Add(\omega,\omega_1)$ is strongly proper for every countable $M\prec H_\theta$.

If $\dot\Q$ is an $\Add(\omega,\omega_1)$-name for a countably closed forcing, then the two-step iteration $\Add(\omega,\omega_1)\ast \dot\Q$ is strongly proper for stationarily many countable models by Lemma \ref{GN}, so the conclusion follows by Proposition \ref{strong}.
\end{proof}

In \cite[Problem 2]{scheepersLindelofIndestructibilityTopological2010}, Scheepers and Tall ask whether adding a single Cohen real renders every ground model Lindel\"of space indestructible. We do not know the answer to this question, but we can show that the strong properness machinery is not enough to solve this problem. Indeed, not only was $\kappa\ge\omega_1$ key in the proof of Lemma \ref{GN}, the statement is false if $\kappa\le\omega$, as the next lemma shows:

\begin{lemma}\label{one_real}
Let $\R:=\Add(\w,1)\ast\dot{\Add}(\w_1,1)$, $\theta$ a sufficiently large regular cardinal, and $M\prec H_\theta$ countable. Then no condition in $\R$ is strongly generic for $M$.
\end{lemma}

\begin{proof}
To see this, fix $(p_0,\dot p_1)\in \R$; we need to argue that $(p_0,\dot p_1)$ is not strongly generic. Setting $\delta=M\cap\w_1$, we may assume upon extending $(p_0,\dot p_1)$ if necessary that $p_0\Vdash \dom(\dot p_1)\supseteq \delta$.

Since $\Add(\omega,1)$ is countable, we can fix  an enumeration $\seq{q_\alpha:\alpha<\delta}$ of the set $\{q\in\Add(\w,1):q\le p_0\}$ such that, for every $q\le p_0$, the set $\{\alpha<\delta:q_\alpha=q\}$ is unbounded in $\delta$.

Fix, for each $\alpha<\delta$, a condition $r_\alpha\le q_\alpha$ and an integer $k_\alpha$ so that $r_\alpha\Vdash \dot p_1(\alpha)=k_\alpha$. Let $D$ be the set
\[
\left\{(t_0,\dot t_1)\in M\cap\R: t_0\perp p_0 \vee \exists \alpha<\delta \left[t_0 \leq r_\alpha \wedge t_0 \Vdash (\alpha\in\dom(\dot t_1)\wedge \dot t_1(\alpha)\neq k_\alpha)\right]\right\}.
\]

\begin{claim}
$D$ is dense in $M\cap\R$
\end{claim}

\begin{proof}
Fix $(s_0,\dot s_1)\in M\cap\R$. We may assume that $s_0\le p_0$. Since $\Add(\w,1)$ has the ccc, there is a $\beta<\delta$ so that $s_0\Vdash \dom(\dot s_1)\subset\beta$. Fix $\alpha\in \delta \setminus (\beta+1)$ so that $q_\alpha=s_0$. Now find a name $\dot s_2\in M$ for an extension of $\dot s_1$ so that $s_0 \Vdash \dot s_2(\alpha)\neq k_\alpha$. Then $(r_\alpha, \dot s_2)\in D$ extends $(s_0,\dot s_1)$.
\end{proof}
Finally, observe that, if $(t_0,\dot t_1)\in D$, then $(t_0,\dot t_1)\perp (p_0,\dot p_1)$. Therefore, $(p_0,\dot p_1)\Vdash D\cap \dot K=\emptyset$, where $\dot K$ is a name for the $(\Add(\omega,1)\ast \dot\Add(\omega_1,1))$-generic filter.
\end{proof}

\section{Games}

Here we prove our main result, which establishes Theorems A and B from the introduction.

\begin{theorem}\label{main}
  Suppose that $(X, \tau)$ is a topological space and player II has a winning strategy in one of the following games:
\begin{itemize}
\item the strong countable fan tightness game for $X$ at $x\in X$;
\item the Rothberger game for $X$;
\end{itemize}
Then this is preserved in any forcing extension of $V$.
\end{theorem}

\begin{proof}
We will prove the theorem for the strong countable fan tightness game, the proof for the Rothberger game is similar and left to the reader.

  Let $\P$ be an arbitrary forcing notion. Given a point $x \in X$, let $\Game_x(X)$ denote the 
  strong countable fan tightness game for $X$ at the point $x$. Recall that $\Game_x(X)$ is a two-player game of length $\omega$. In each round, player I first plays a set $A_n \subseteq X$ such that $x \in \mathrm{cl}(A_n)$; then player II plays a point $x_n \in A_n$. Player II wins if $x \in \mathrm{cl}(\{x_n :n <\omega\})$; otherwise, player I wins.
  
  Let $\sigma_x$ be a winning strategy for player II in $\Game_x(X)$. Concretely, 
  $\sigma_x$ is a function such that $\dom(\sigma_x)$ is the set of all finite sequences 
  $\vec{A} = \langle A_k :k \leq n \rangle$ such that each $A_k$ is a subset of $X$ with 
  $x \in \mathrm{cl}(A_k)$. For each such $\vec{A}$, we have $\sigma(\vec{A}) \in A_n$.
  
  We aim to show that player II continues to have a winning strategy in $\Game_x(X)$ in 
  $V^{\P}$. Suppose that $G \subseteq \P$ is generic over $V$, and move to $V[G]$. We will 
  define a winning strategy $\sigma^*$ for player II in $\Game_x(X)$ by recursion on the 
  length of the play. We will ensure that, for a sequence $\vec{A} = \langle A_k :
  k \leq n \rangle$ such that each $A_k$ is a subset of $X$ with $x \in \mathrm{cl}(A_k)$, 
  the strategy outputs not only player II's play $\sigma^*(\vec{A})$ but also auxiliary 
  objects of the following type:
  \begin{itemize}
    \item a $\P$-name $\dot{A}_n \in V$ such that $(\dot{A}_n)_G = A_n$;
    \item a condition $p_n \in G$;
    \item $A^*_n := \{y \in X :\exists p' \leq p_n \ p' \Vdash_{\P} y \in \dot{A}_n\}$.
  \end{itemize}
  To help with the selection of these objects, fix in $V$ a sufficiently large regular cardinal 
  $\theta$ and a wellordering $\vartriangleleft$ of $H(\theta)$.
  
  Suppose that we are given a finite sequence $\vec{A} = \langle A_k :k \leq n \rangle$ 
  of subsets of $X$ such that $x \in \mathrm{cl}(A_k)$ for all $k \leq n$, and suppose that 
  we have already specified $\sigma^*(\vec{A} \restriction m)$ for all $m < n$, as well 
  as auxiliary objects $\dot{A}_k$, $p_k$, and $A^*_k$ for all $k < n$. Assume that we have 
  arranged so that $x \in \mathrm{cl}(A^*_k)$ for all $k < n$.
  
  First, let $\dot{A}_n$ be the $\vartriangleleft$-least $\P$-name for a subset of $X$ such that 
  $(\dot{A}_n)_G = A_n$. Let $q_n$ be the $\vartriangleleft$-least element of $G$ such that, in 
  $V$, we have $q_n \Vdash_{\P} x \in \mathrm{cl}(\dot{A}_n)$. For each $p \leq q_n$, 
  let $A^*_{n,p}$ be the set of all $y \in X$ for which there exists $p' \leq p$ such 
  that $p' \Vdash_{\P} y \in \dot{A}_n$. Note that $A^*_{n,p} \in V$ and 
  $x \in \mathrm{cl}(A^*_{n,p})$.
  
\begin{claim}\label{4.2}
There exists $p \leq q_n$ in $G$ such that 
  $\sigma_x(\langle A^*_k :k < n \rangle ^\frown \langle A^*_{n,p}\rangle) \in A_n$.
\end{claim} 

\begin{proof}
Suppose otherwise. We can then fix a condition $p^*\in G$ such that $p^*\le q_n$ and
\[
p^*\Vdash \neg\exists p\le q_n \left(p\in\dot G\wedge \sigma_x(\langle A^*_k :k < n \rangle ^\frown \langle A^*_{n,p}\rangle) \in \dot{A}_n\right).
\]
Set $y:=\sigma_x(\langle A^*_k :k < n \rangle ^\frown \langle A^*_{n,p^*}\rangle)$. Since $y\in A^*_{n,p^*}$, we can find $p^{**}\le p^*$ such that $p^{**}\Vdash y\in\dot A_n$.

Let $G^*$ be $\P$-generic over $V$ with $p^{**}\in G^*$. Since $p^{**}\le p^*$, it follows that $p^*\in G^*$, hence
\[
V[G^*]\models \neg\exists p\le q_n \left(p\in G^*\wedge \sigma_x(\langle A_k^*:k<n\rangle^\frown \langle A_{n,p}^*\rangle)\in (\dot{A}_n)_{G^*}\right).
\]
On the other hand, $p^{**}\Vdash y\in \dot A_n$, hence
\[
V[G^*]\models \left(p^*\le q_n \wedge p^*\in G^*\wedge \sigma_x(\langle A_k^*:k<n\rangle^\frown \langle A_{n,p^*}^*\rangle)\in (\dot{A}_n)_{G^*}\right).
\]
This is a contradiction.
\end{proof}

Let $p_n$ be the $\vartriangleleft$-least $p$ satisfying the claim, and let 
  $A^*_n := \{y \in X :\exists p' \leq p_n  (p' \Vdash_{\P} y \in \dot{A}_n)\}$. 
  Finally, set $\sigma^*(\vec{A}) = \sigma(\langle A^*_k :k \leq n \rangle)$.
  
  In $V$, let $\dot{\sigma}^*$ be a name for the strategy defined above. 
  We claim that $\dot{\sigma}^*$ is forced to be a winning strategy for player II. 
  Suppose not, and fix in $V$ a sequence $\langle \dot{B}_n :n < \omega \rangle$ 
  and a condition $r \in \P$ such that it is forced by $r$ that
  \begin{itemize}
    \item for all $n < \omega$, $\dot{B}_n$ is a subset of $X$ and $x \in 
    \mathrm{cl}(\dot{B}_n)$;
    \item if $\langle \dot{x}_n :n < \omega \rangle$ is the sequence produced by player II 
    playing $\dot{\sigma}^*$ against $\langle \dot{B}_n :n < \omega \rangle$, 
    then $x \notin \mathrm{cl}(\{\dot{x}_n :n < \omega\})$.
  \end{itemize}
  By extending $r$ further if necessary, we can in fact assume that there is an open set 
  $U \ni x$ such that $r \Vdash_{\P} \{\dot{x}_n :n < \omega\} \cap U = \emptyset$.
  
  For $n < \omega$, let $\dot{a}_n$, $\dot{p}_n$, and $\dot{A}^*_n$ 
  be $\P$-names forced by $r$ to be equal 
  to the auxiliary objects\footnote{We caution the reader that $\dot a_n$ is a name for a name.} $\dot{A}_n$, $p_n$, and $A^*_n$ produced during the run of 
  $\Game_x(X)$ in which player I plays $\langle \dot{B}_n :n < \omega \rangle$ and 
  player II plays according to $\dot{\sigma}^*$. Recursively define a decreasing sequence 
  $\langle r_n :n < \omega \rangle$ of conditions in $\P$ such that, for all $n < \omega$, 
  we have
  \begin{itemize}
    \item $r_n \leq r$;
    \item $r_n$ decides the values of $\dot{a}_n$, $\dot{p}_n$, and $\dot{A}^*_n$, say as 
    $\dot{A}_n$, $p_n$, and $A^*_n$;
    \item $r_n \leq p_n$.
  \end{itemize}
  For each $n < \omega$, note that $A^*_n = \{y \in X :\exists p' \leq p_n \ p' \Vdash_{\P} 
  y \in \dot{A}_n \}$. Now let $\langle x_n :n < \omega \rangle$ be the plays made by player II 
  in a run of $\Game_x(X)$ in which player I plays $\langle A^*_n :n < \omega \rangle$ and 
  player II plays according to $\sigma_x$. Because $\sigma_x$ is a winning strategy for player II, 
  we have $x \in \mathrm{cl}(\{x_n :n < \omega\})$; in particular, there is $n < \omega$ 
  such that $x_n \in U$. Now note that, by our definition of $\dot{\sigma}^*$, $r_n$ forces that 
  $\langle x_k :k \leq n \rangle$ is the sequence of the first $(n+1)$-many moves made 
  by player II in the run of $\Game_x(X)$ in which player I plays $\langle \dot{B}_n :n < 
  \omega \rangle$ and player II plays $\dot{\sigma}^*$; i.e., for all $k \leq n$, we have 
  $r_n \Vdash_{\P} \dot{x}_k = x_k$. Then $x_n \in U$ contradicts the fact that 
  $r_n \leq r$ and $r \Vdash_{\P} \{\dot{x}_i :i < \omega\} \cap U = \emptyset$.
\end{proof}

\begin{remark}
If we try to extend Theorem \ref{main} to the countable fan tightness case, we run into trouble with the proof of Claim \ref{4.2}, because player II makes more than once choice in each round: if we let $F:=\sigma_x(\seq{A_k^*:k<n}^\frown \seq{A_{n,p^*}^*})$, then by $F\subset A_{n,p^*}^*$ we can find, for each $y\in F_n$, some $p_y \le p^*$ with $p_y \Vdash y\in \dot A_n$, but there is no reason to believe that $\{p_y:y\in F\}$ has a lower bound. In fact, Theorem \ref{main} fails for the countable fan tightness game: let $X={}^{\omega_1}2$ with the product topology, and let $Y=C_p(X)$. By \cite[Example 4]{scheepersRemarksCountableTightness2014}, player II has a winning strategy in the countable fan tightness game for $Y$, but $Y$ is not countably tight in $V^{\Add(\omega_1,1)}$. Therefore, by Lemma \ref{game}, in $V^{\Add(\omega_1,1)}$, player II does not have a winning strategy in the countable fan tightness game for $Y$.

Theorem \ref{main} also fails for the Menger game. Again consider the space $X={}^{\omega_1}2$ with the product topology. By Tychonoff's theorem, $X$ is compact, so player II has a very simple winning strategy in the Menger game: if player I's play at stage $0$ is $\mathcal{U}_0$, then player II chooses a finite subcover $\mathcal{F}_0$, and plays arbitrarily at every other stage of the game. On the other hand, by Example \ref{X}, $X$ is not Lindel\"of in $V^{\Add(\omega_1,1)}$, hence not Menger and, by Lemma \ref{game}, player II does not have a winning strategy for the Menger game in $V[G]$.
\end{remark}

\section{Open questions}

While Theorems A and B are both quite general, referring to generic extensions by arbitrary forcings, the proofs use the forcing machinery in a very strong way, so it seems unlikely that a minor modification of the previous arguments will establish the analogous results for an arbitrary pair of transitive models $V\subset W$. This suggests the following questions:

\begin{question}
Assume that $0^\sharp$ exists (or even that $\PFA$ holds). Is there a topological space $(X,\tau)\in L$ so that for some $x\in X$, in $L$, player II has a winning strategy in the strong countable fan tightness game, but II has no such strategy in the game as played in the universe $V$?
\end{question}

\begin{question}
Assume that $0^\sharp$ exists (or even that $\PFA$ holds). Is there a topological space $(X,\tau)\in L$ so that, in $L$, player II has a winning strategy in the Rothberger game, but II has no such strategy in the game as played in the universe $V$?
\end{question}

\printbibliography

\end{document}